\documentclass[10pt]{amsart}
\usepackage[utf8]{inputenc}
\usepackage{amscd}
\usepackage{amssymb}
\usepackage{hyperref}
\usepackage{pstricks}
\usepackage{empheq}
\usepackage{graphicx}
\usepackage{float}
\usepackage{todonotes}
\usepackage{subcaption}
\usepackage{xcolor}
\usepackage{textcomp}
\usepackage{cancel}
\usepackage[mathlines]{lineno}
\usepackage{comment}
\usepackage{ulem}
\newtheorem{theorem}{Theorem}
\newtheorem{proposition}{Proposition}
\newtheorem{lemma}{Lemma}
\newtheorem{definition}{Definition}
\newtheorem{corollary}{Corollary}

\newtheorem{remark}{Remark}

\numberwithin{equation}{section}
\numberwithin{lemma}{section}

\title[]{Conditions for uniform $h$--dichotomy in terms of uniform non criticality, expansiveness and via generalized Floquet theory}

\author[Elorreaga]{Heli Elorreaga}
\author[Robledo]{Gonzalo Robledo}
\author[Urrutia]{David Urrutia--Vergara}
\address{Departamento de Matem\'aticas, Universidad de Chile, Casilla 653, Santiago, Chile.}
\address{Departamento de Matem\'atica, Facultad de Ciencias, Universidad del Bio--Bio, Casilla 5-C, Concepci\'on, Chile}
\email{helorreaga@ubiobio.cl,grobledo@uchile.cl,david.urrutia@ug.uchile.cl}
\thanks{This research has been supported by the Chilean Agency for Research and Development, ANID-Chile through the grant FONDECYT Postdoctorado project 3240354 (H. Elorreaga), FONDECYT Regular project 1240361 (G. Robledo) and the doctoral scholarship 21241939 (D. Urrutia--Vergara).}

\begin{document}

\begin{abstract}
In this article, we complete the study of the equivalences between the properties of $h$--dichotomy, $h$--noncriticality and $h$--expansiveness
of a linear nonautonomous ODE system which had been initiated in a previous work. Moreover, we extend a result of the generalized Floquet theory developed by T.A. Burton and J.S. Muldowney by providing a necessary and sufficient condition for $h$--dichotomy. It should be noted that all the results have been obtained by using a characterization of the $h$--dichotomy by a group theory approach recently developed by J.F. Pe\~na and S. Rivera--Villagr\'an.
\end{abstract}

\maketitle

\section{Introduction}
The uniform $h$--\textit{dichotomy} is a property of non autonomous linear systems of differential
and difference equations which; roughly speaking; splits any non trivial solution in a contractive and an expansive
part which are dominated by a growth rate $h(\cdot)$ that will be formally described later.

In the recent years  we have witnessed a renewed interest in studying $h$-dichotomies \cite{CPR,DD,E-P-R,PV,Silva,WXZ} and this work revisits the uniform $h$--dichotomy
for non  autonomous linear systems of differential equations: 
\begin{equation}
\label{lin}
x'=A(t)x  \quad \textnormal{for any $t \in J:=(a_{0},+\infty)$},
\end{equation}
where $a_{0} \in \mathbb{R}\cup \{-\infty\}$ and the map $t\mapsto A(t)\in M_{n}(\mathbb{R})$ is continuous on $J$. From now on, a fundamental matrix of (\ref{lin}) will be denoted by $\Phi(t)$, while the corresponding transition matrix will be $\Phi(t,s)$. On the other hand, given
any norm $|\cdot|$ on $\mathbb{R}^{n}$, its induced matrix norm will be denoted by $||\cdot||$.

\begin{definition}
\label{GR}
A growth rate $h(\cdot)$ is a strictly increasing homeomorphism $h\colon J=(a_{0},+\infty)\to (0,+\infty)$. In addition, if $h$ is differentiable on $J$, we will say that $h$ is a differentiable growth rate.
\end{definition}

A formal definition of the uniform $h$--dichotomy is given by:
\begin{definition}
\label{UHD}
Given a growth rate $h\colon (a_{0},+\infty)\to (0,+\infty)$ and an interval $\mathcal{I}\subseteq J$, the linear system \eqref{lin} has a uniform $h$--dicohotomy on $\mathcal{I}$
if there exists a projector $P(t)$ and a pair of constants $K\geq 1$ and $\alpha>0$ such that
\begin{equation}
\label{Invariance}
P(t)\Phi(t,s)=\Phi(t,s)P(s) \quad \textnormal{for any $t,s\in \mathcal{I}$} ,
\end{equation}
and
\begin{equation}
\label{HD}
\left\{\begin{array}{rcl}
||\Phi(t,s)P(s)||  &  \leq & \displaystyle  K\left(\frac{h(t)}{h(s)}\right)^{-\alpha} \quad \textnormal{for any $t\geq s$ with $t,s\in \mathcal{I}$} \\\\
||\Phi(t,s)Q(s)||  &  \leq & \displaystyle K\left(\frac{h(s)}{h(t)}\right)^{-\alpha} \quad \textnormal{for any $s\geq t$ with $t,s\in \mathcal{I}$},
\end{array}\right.
\end{equation}
where $Q(\cdot)$ verifies $P(t)+Q(t)=I$ for any $t\in \mathcal{I}$.
\end{definition}

There exists an alternative characterization for the uniform $h$--dichotomy on $\mathcal{I}$ which is
carried out in terms of constants projectors. We refer to \cite[Prop.2]{E-P-R} for additional details:
\begin{proposition}
\label{ProyCte}
The system \eqref{lin} has a uniform $h$--dichotomy on $\mathcal{I}$ if and only if, given a fundamental matrix $\Phi(t)$, there exists a constant projector $P$ and constants $K\geq 1$ and $\alpha>0$
such that:
\begin{equation}
\label{eq:2.2}
\left\{\begin{array}{rccr}
||\Phi(t)P\Phi^{-1}(s)||&\leq K \left( \frac{h(t)}{h(s)}\right)^{-\alpha} & \textnormal{if} & t\geq s,\quad t,s\in \mathcal{I}\\
||\Phi(t)(I-P)\Phi^{-1}(s)|| &\leq K \left( \frac{h(s)}{h(t)}\right)^{-\alpha} & \textnormal{if} & t \leq s, \quad t,s \in \mathcal{I}.
\end{array}\right.
\end{equation}
\end{proposition}

The uniform $h$--dichotomies generalizes the exponential dichotomy introduced by O. Perron \cite{Perron}, 
which considers the growth rate $h(t)=e^{t}$ and $J=\mathbb{R}$. In fact, R. Martin Jr. \cite{Martin}, J.S. Muldowney \cite{Muldowney} and M. Pinto $\&$ R. Naulin \cite{Pinto92,NP94} considered different growth rates.

In general, the study of $h$-dichotomies has been carried out in an attempt to emulate the results already known for the exponential dichotomy such as the topological conjugacy between the solutions of (\ref{lin}) with a quasilinear perturbation \cite{Reinfelds}, robustness of the dichotomy under small perturbations \cite{NP95,ZFY} and spectral theories arising from the dichotomy \cite{Silva,WPX}. 

The main goals of this article are two: the first one is to achieve the study initiated in \cite{E-P-R} where equivalences between uniform $h$--dichotomy with the properties of \textit{noncritical uniformity}
and \textit{expansiveness} (formal definitions will be given later) was established for the case $\mathcal{I}=[h^{-1}(1),+\infty)$ and now will be stated for $\mathcal{I}=J$. The second objective of this article is to obtain necessary and sufficient conditions for uniform $h$--dichotomy in the framework of the generalized Floquet theory developed by T.A. Burton and J.S. Muldowney in \cite{Burton}.

The section 2 reviews basic properties and recent results about the uniform $h$--dichotomy \cite{CPR,E-P-R,PV}. 
The section 3 provides a complete study of the equivalences between the properties of uniform $h$--dichotomy,
uniform $h$--noncriticality and $h$--expansiveness for the case $\mathcal{I}=J$. The section 4 expands some results of the generalized Floquet
theory and provides a criterion for uniform $h$--dichotomy on $J$.

\textit{Notation:} Given a linear operator $L$ its range will be denoted by $\mathcal{R}L$ whereas its kernel
will be denoted by $\mathcal{N}L$. As usual, the identity matrix will be denoted by $I$, the dimension of
a vector subspace $W\subset \mathbb{R}^{n}$ will be denoted by $\textnormal{Dim}\, W$ while
the direct sum of two subspaces is described by $\oplus$.

\section{The uniform $h$--dichotomy:  a short review}
Given any differentiable
growth rate $h(\cdot)$ and $\alpha>0$ it is easy to construct illustrative examples of systems having a uniform $h$--dichotomy as follows:

\begin{equation}
\label{exemple}
\left(\begin{array}{c}
\dot{x}_{1}\\
\dot{x}_{2}
\end{array}\right)=
\left[\begin{array}{cc}
\displaystyle -\alpha \frac{h'(t)}{h(t)}  &  0 \\
0 & \displaystyle \alpha \frac{h'(t)}{h(t)}
\end{array}\right]
\left(\begin{array}{c}
x_{1}\\
x_{2}
\end{array}\right) \quad \textnormal{for any $t\in J$}.
\end{equation}

In fact, notice that (\ref{Invariance}) and (\ref{HD}) are verified with $K=1$ and
\begin{equation}
\label{dugma}
\Phi(t)=
\left[\begin{array}{cc}
\frac{1}{h(t)^{\alpha}}  &  0 \\
0 & h(t)^{\alpha}
\end{array}\right]  \quad \textnormal{and} \quad
P(t)=\left[\begin{array}{cc}
1 & 0 \\
0 & 0
\end{array}\right] \quad \textnormal{for any $t\in J$}.
\end{equation}

Now, notice that if the linear system (\ref{lin}) has uniform $h$--dichotomy on $\mathcal{I}\subset J$ it is
straightforward to see that any nontrivial solution $t\mapsto x(t,t_{0},x_{0})$ of (\ref{lin}) passing through $x_{0}$ at $t=t_{0}\in \mathcal{I}$ can be decomposed as follows:
\begin{equation}
\label{splitting}
x(t,t_{0},x_{0})=\Phi(t,t_{0})x_{0}=\underbrace{\Phi(t,t_{0})P(t_{0})x_{0}}_{:=x^{+}(t,t_{0},x_{0})}+\underbrace{\Phi(t,t_{0})Q(t_{0})x_{0}}_{:=x^{-}(t,t_{0},x_{0})}.
\end{equation}

The following result, stated without proof in \cite{E-P-R}, is a direct consequence from $h$--dichotomy on $\mathcal{I}$ combined with (\ref{splitting}):
\begin{lemma}
\label{SPLIT}
If \eqref{lin} has a uniform $h$--dichotomy on $\mathcal{I}$, then any non trivial solution of \eqref{lin} has the following forward and backward behavior:
\begin{itemize}
\item[a)] If $t,t_{0}\in \mathcal{I}$ and $t\geq t_{0}$, the maps $x^{+}(\cdot)$ and $x^{-}(\cdot)$ have the following estimations: 
\begin{equation}
\label{ae1}
|x^{+}(t,t_{0},x_{0})|\leq K\left(\frac{h(t_{0})}{h(t)}\right)^{\alpha}|x_{0}|  \,\, \textnormal{and} \,\,\,
\frac{|Q(t_{0})x_{0}|}{K} \left( \frac{h(t)}{h(t_{0})}\right)^{\alpha}\leq |x^{-}(t,t_{0},x_{0})|,
\end{equation}
\item[b)] If $t,t_{0}\in \mathcal{I}$ and $t\in (a_{0},t_{0})$, the maps $x^{+}(\cdot)$ and $x^{-}(\cdot)$ have the following estimations:
\begin{equation}
\label{ae2}
\left(\frac{h(t_{0})}{h(t)}\right)^{\alpha}\frac{|P(t_{0})x_{0}|}{K} \leq |x^{+}(t,t_{0},x_{0})| \quad \textnormal{and} \quad
|x^{-}(t,t_{0},x_{0})|\leq K\left(\frac{h(t)}{h(t_{0})}\right)^{\alpha}.
\end{equation}
\end{itemize}
\end{lemma}

\begin{proof}
The left inequality of (\ref{ae1}) and the right inequality of (\ref{ae2}) follows directly from the inequalities (\ref{HD}).  
Now, if $t\geq t_{0}$, by using (\ref{Invariance}) combined with $Q^{2}(t_{0})=Q(t_{0})$ it follows that
\begin{displaymath}
\begin{array}{rcl}
Q(t_{0})x_{0}&=&\Phi(t_{0},t)\Phi(t,t_{0})Q^{2}(t_{0})x_{0}\\\\
&=&\Phi(t_{0},t)Q(t)\Phi(t,t_{0})Q(t_{0})x_{0}\\\\
&=&\Phi(t_{0},t)Q(t)x^{-}(t,t_{0},x_{0}),
\end{array}
\end{displaymath}
therefore by (\ref{HD}) we can deduce that
$$
|Q(t_{0})x_{0}|\leq K\left(\frac{h(t)}{h(t_{0})}\right)^{-\alpha}|x^{-}(t,t_{0},x_{0})|,
$$
and the right inequality of (\ref{ae1}) follows. Finally, the left inequality of (\ref{ae2}) can be proved similarly
by noticing that
$$
P(t_{0})x_{0}=\Phi(t_{0},t)P(t)x^{+}(t,t_{0},x_{0}).
$$
\end{proof}

\begin{lemma}
\label{FDIA}
If the linear system \eqref{lin} has an uniform $h$--dichotomy on $J$, the unique bounded solution is the trivial one.
\end{lemma}

\begin{proof}
Let $t\mapsto x(t):=x(t,t_{0},x_{0})$ be a nontrivial solution of (\ref{lin}) passing through $x_{0}$ at $t=t_{0}$. We 
can prove that this solution cannot be bounded on $J$. In fact, by Lemma \ref{SPLIT} we know that this solution can be splitted
as described in (\ref{splitting}).  Moreover, by (\ref{ae1}) we have that $x^{+}(t,t_{0},x_{0})\to 0$
and $|x^{-}(t,t_{0},x_{0})|\to +\infty$ when $t\to +\infty$, thus $t\mapsto x(t)$ cannot be bounded on $[t_{0},+\infty)$. Similarly, by (\ref{ae2})
we also have that $|x^{+}(t,t_{0},x_{0})|\to + \infty$
and $x^{-}(t,t_{0},x_{0})\to 0$ when $t\to a_{0}^{+}$, therefore $t\mapsto x(t)$ cannot be bounded on $(a_{0},t_{0}]$ and the Lemma follows.
\end{proof}

\subsection{Uniform $h$--dichotomies from a group theory perspective}
An innovative approach introduced by J.F. Pe\~na and S. Rivera--Villagr\'an in \cite{PV} proved that, for any growth rate $h\colon J\to (a_{0},+\infty)$, 
a topological and totally ordered abelian group $(J,*)$ can be constructed by considering the following operation:
\begin{equation}
\label{LCI}
\begin{array}{rcl}
J \times J & \to &  J  \\
(t,s) &\mapsto &  t\ast s:=h^{-1}\left(h(t)h(s)\right).
\end{array} 
\end{equation}
Moreover, the unit element and the inverse for any $t\in J$ are respectively defined by
\begin{equation}
\label{N-I}
e_{\ast}:=h^{-1}(1) \quad \textnormal{and} \quad
t^{\ast -1}:=h^{-1}\left(\frac{1}{h(t)}\right).
\end{equation}

As any abelian group is a $\mathbb{Z}$--module by considering
the external composition law
\begin{equation*}
%\label{LCE-bis}
\begin{array}{rcl}
\mathbb{Z} \times \mathbb{R} & \to &  \mathbb{R}  \\
(k,t) &\mapsto &  t^{*k}:=\left\{\begin{array}{rcl}
 \underbrace{t \ast \cdots \ast t}_{k-\textnormal{times}} &\textnormal{if}& k>0 \\
 e_{\ast}  &\textnormal{if}&  k=0 \\
 \underbrace{t^{\ast -1} \ast \cdots \ast t^{\ast -1}}_{k-\textnormal{times}} &\textnormal{if}& k<0.
 \end{array}\right.
\end{array} 
\end{equation*}

By using recursively the identities (\ref{N-I}) combined with the fact that
$h(\cdot)$ is an homeomorphism we can see that
$$
t^{*k}=h^{-1}(h(t^{*k}))=h^{-1}(h(t)^{k}) \quad \textnormal{for any $k\in \mathbb{Z}$},
$$
and the above external composition law can be revisited as
\begin{equation}
\label{LCE-bis2}
\begin{array}{rcl}
\mathbb{Z} \times J & \to &  J  \\
(k,t) &\mapsto &  t^{*k}=h^{-1}\left(h(t)^{k}\right),
\end{array} 
\end{equation}
which leads to the property
\begin{equation}
\label{LP}
h(t^{*k})=h(t)^{k}  \quad \textnormal{for any $k\in \mathbb{Z}$}.
\end{equation}

The following identities for any pair $t,s\in J$ are stated in \cite{E-P-R}:
\begin{subequations}
  \begin{empheq}[left=\empheqlbrace]{align}
&
h(t\ast s)=h(t)h(s) \quad \textnormal{and} \quad h(s^{\ast -1})=\frac{1}{h(s)}, \label{group0} \\
&
(t\ast s)^{\ast -1}=h^{-1}\left(\frac{1}{h(t\ast s)} \right)=h^{-1}\left(\frac{1}{h(t)h(s)} \right), \label{group1} \\
& t\ast s^{\ast -1}=t\ast h^{-1}\left(\frac{1}{h(s)}\right)= h^{-1}\left(\frac{h(t)}{h(s)}\right), 
\label{group2} \\
&
(t\ast s^{\ast -1})^{\ast -1}=h^{-1}\left(\frac{1}{h(t)h(s^{\ast -1})}\right)=h^{-1}\left(\frac{h(s)}{h(t)}\right)=s\ast t^{\ast-1}. \label{group3}
 \end{empheq}
 \end{subequations}

In addition, in \cite{E-P-R} it was noticed that the group $(J,\ast)$ is also totally ordered,
where the order $\leq_{\ast}$ on $J$:
\begin{equation}
\label{order}
s\leq_{\ast} t  \quad \textnormal{if and only if} \quad   e_{\ast} \leq t\ast s^{\ast -1}   
\end{equation}
verifies that 
\begin{equation}
\label{order2}
t \leq_{\ast} s \quad \textnormal{if and only if} \quad t\leq s \quad \textnormal{for any $t,s\in J$},
\end{equation}

It was also noticed in \cite{E-P-R} that, as the group is ordered, the following absolute value $|\cdot|_{\ast}\colon J\to [e_{\ast},+\infty)$ can be introduced:
\begin{equation}
\label{abs}
|t|_{\ast}=\left\{\begin{array}{ccr}
t &\textnormal{if}& e_{*}\leq t, \\
t^{\ast -1} &\textnormal{if}& \,\, t<e_{\ast},
\end{array}\right.
\end{equation}
moreover, as $(J,\ast)$ is abelian, the triangle inequality is satisfied:
\begin{equation}
|t\ast s|_{\ast} \leq |t|_{\ast} \ast |s|_{\ast}.
\end{equation}

These facts allow to define a distance $d \colon (J,\ast) \times (J,\ast) \to   ([e_{\ast},+\infty),\ast)$
as follows:
\begin{equation}
\label{distance}
d(t,s):=|t\ast s^{\ast -1}|_{\ast}.
\end{equation}

It will be useful to note that, given $e_{\ast}<L$ it follows that
\begin{equation}
\label{EVA}
|t\ast s^{\ast -1}|_{\ast}\leq L  \iff L^{\ast-1}\leq  t\ast s^{\ast -1}\leq L \iff \frac{1}{h(L)}\leq \frac{h(t)}{h(s)} \leq h(L) .
\end{equation}

As pointed out in \cite[Sect.2]{E-P-R}, the group $(J,\ast)$ and his properties makes possible two alternative characterizations for
the uniform $h$--dichotomy. Firstly, the property (\ref{group2}) allow us to see the uniform $h$--dichotomy from the following perspective:
\begin{equation}
\label{alt-hd}
\left\{\begin{array}{rcl}
||\Phi(t,s)P(s)||  &  \leq & \displaystyle  Kh(t\ast s^{\ast -1})^{-\alpha} \quad \textnormal{for any $t\geq s$  with $t,s\in \mathcal{I}$}, \\\\
||\Phi(t,s)Q(s)||  &  \leq & \displaystyle K h(s\ast t^{\ast -1})^{-\alpha} \quad \textnormal{for any $s\geq t$ with $t,s\in \mathcal{I}$}.
\end{array}\right.
\end{equation}

Secondly, notice that (\ref{order}) combined (\ref{distance}) allow to see uniform $h$--dichotomy from the alternative perspective:
\begin{equation}
\label{alt-hd2}
\left\{\begin{array}{rcl}
||\Phi(t,s)P(s)||  &  \leq & \displaystyle  Kh( d(t,s))^{-\alpha} \quad \textnormal{for any $t\geq s$ with $t,s\in \mathcal{I}$} \\\\
||\Phi(t,s)Q(s)||  &  \leq & \displaystyle K h(d(t,s))^{-\alpha} \quad \textnormal{for any $s\geq t \geq a$ with $t,s\in \mathcal{I}$}.
\end{array}\right.
\end{equation}

\subsection{Uniform bounded growth}
A careful reading of the dichotomies literature shows its strong relation with the properties of bounded growth and decay
which, tailored for an specific growth rate $h(\cdot)$, are defined as follows: 
\begin{definition}
Given a growth rate $h\colon J\to (0,+\infty)$ and its corresponding group $(J,*)$, the linear system \eqref{lin} has:
\begin{itemize}
\item[a)]  A uniform \textbf{bounded $h$--growth} on $\mathcal{I}\subseteq J$ if for each $T>e_{\ast}$ there exists
$C_{T}\geq 1$ such that any solution $t\mapsto x(t)$ verifies 
\begin{equation}
\label{BG}    
|x(t)|\leq C_{T}|x(s)| \quad \textnormal{for any $t\in [s,s\ast T]\cap \mathcal{I}$}.
\end{equation}
\item[b)] A  uniform \textbf{bounded $h$--decay} on $\mathcal{I}\subseteq J$ if for each $T>e_{\ast}$ there exists $C_T\geq1$ such that any solution  $t\mapsto x(t)$ verifies
    \begin{equation}\label{BD}
        |x(t)|\leq C_T|x(s)| \quad \textnormal{for any $t\in[s\ast T^{\ast-1}, s]\cap \mathcal{I}$}.
    \end{equation}
\item[c)]  A uniform \textbf{bounded $h$--growth and $h$--decay} on $\mathcal{I}\subseteq J$ if for each $e_{\ast}<T$ there exists $C_T\geq1$ such that any solution  $t\mapsto x(t)$ verifies
    \begin{equation}\label{BGD}
        |x(t)|\leq C_T|x(s)| \quad \textnormal{for any $t\in[s\ast T^{\ast-1}, s\ast T]\cap \mathcal{I}$}.
    \end{equation}  
\end{itemize}
\end{definition}

The characterization (\ref{LCE-bis2}) allow us to
make a partition of $\mathcal{I}$
$$
\mathcal{I}=\bigcup\limits_{k\in \mathbb{Z}}I_{k}  \quad \textnormal{where $I_{k}=[s\ast T^{\ast (k-1)},s\ast T^{\ast k})$},
$$
and, as noticed in \cite{E-P-R} by using (\ref{abs}) combined with (\ref{distance}) and $e_{\ast}<T$, we can verify that the above partition of $J$ is uniform since is composed by intervals of large $T$:
\begin{displaymath}
d(s\ast T^{\ast k},s\ast T^{\ast (k-1)})=d(T\ast s\ast T^{\ast (k-1)},s\ast T^{\ast (k-1)})=T.  
\end{displaymath}

The above mentioned uniform partition was key to prove the following result:
\begin{proposition}\cite[Lemmas 3.1,3.2,3.3]{E-P-R}
The linear system \eqref{lin} has:
\begin{itemize}
    \item[a)] A uniform bounded $h$--growth on $\mathcal{I}$ if and only if there exist 
$K_{0}\geq 1$ and $\beta \geq 0$ such that
\begin{equation}
\label{equivalencia}
||\Phi(t,s)|| \leq K_{0}\left(\frac{h(t)}{h(s)}\right)^{\beta}=K_{0}h(t\ast s^{\ast -1})^{\beta} \quad \textnormal{for any $t\geq s$ with $t,s\in \mathcal{I}$.}
\end{equation}
    \item[b)]  A uniform bounded $h$--decay on $\mathcal{I}$ if and only if there exist 
$K_{0}\geq 1$ and $\beta\geq0$ such that
\begin{equation}
\label{equivalenciaBD}
||\Phi(t,s)|| \leq K_{0}\left(\frac{h(s)}{h(t)}\right)^{\beta} \quad \textnormal{for any $s\geq t$ with $t,s\in \mathcal{I}$.} 
\end{equation}
\item[c)]   A uniform bounded $h$--growth and $h$--decay on $\mathcal{I}$ if and only if there exist 
$K_{0}\geq 1$ and $\beta\geq0$ such that
\begin{equation}
\label{equivalenciaBGD}
||\Phi(t,s)|| \leq K_{0}h(|t\ast s^{\ast-1}|_{\ast})^{\beta} \quad \textnormal{for any $t, s\in \mathcal{I}$}. 
\end{equation}
\end{itemize}
\end{proposition}

\subsection{Uniform $h$--noncriticality and $h$--expansiveness}
In this subsection we will recall two properties which are essential
to have alternative characterizations of the $h$--dichotomy.

\begin{definition}
\label{NCUTR}
The system \eqref{lin} is uniformly $h$--noncritical on an interval $\mathcal{I}\subseteq J$ if there exists $T>e_{\ast}$
and $\theta\in (0,1)$ such that any solution $t\mapsto x(t)$ of \eqref{lin} satisfies
    \begin{equation}
    \label{NCP}
    \begin{array}{rcl}
        |x(t)|\leq \theta\sup\limits_{|u\ast t^{\ast-1}|_{\ast}\leq T}|x(u)| &= & \theta\sup\{|x(u)|\colon T^{\ast-1}<u\ast t^{\ast-1}\leq T \}\\\\
        &=&  \theta\sup\{|x(u)|\colon t\ast T^{\ast-1}<u\leq t\ast T \},
    \end{array}
    \end{equation}
    for all $t$ such that $[t\ast T^{\ast-1}, t\ast T]\subset \mathcal{I}$.
\end{definition}

It will be useful to note that:
\begin{itemize}
\item[$\bullet$] If $\mathcal{I}=J^{+}:=[e_{\ast},+\infty)$, the inequality (\ref{NCP}) must be verified for any $t\geq T$, which coincides with the Definition stated in \cite{E-P-R}.
\item[$\bullet$] If $\mathcal{I}=J^{-}:=(a_{0},e_{\ast}]$, the inequality (\ref{NCP}) must be verified for any  $t\in (a_{0}\ast T,T^{\ast-1}]$.
\item[$\bullet$] If $\mathcal{I}=J$, the inequality (\ref{NCP}) must be verified for any $t\in J$.
\end{itemize}

\begin{remark}
\label{ANCU}
It direct to see that if \eqref{lin} is uniformly $h$--noncritical on $J$, then it is also uniformly
noncritical on $(a_{0},e_{*}]$ and $[e_{*},+\infty)$. In addition, it is useful to note that \eqref{NCP} also can be written as follows;
\begin{displaymath}
     |x(t)|\leq \theta\sup\limits_{|u\ast t^{\ast-1}|_{\ast}\leq T}|x(u)| =  \theta\sup\{|x(u)|\colon \frac{1}{h(T)}<\frac{h(u)}{h(t)}\leq h(T) \}
\end{displaymath}
for all $t$ such that $[t\ast T^{\ast-1}, t\ast T]\subset \mathcal{I}$.
\end{remark}

\begin{definition}
\label{NCUTH}
    The system \eqref{lin} is $h$--expansive on an interval $\mathcal{I}\subseteq J$ if there exists  positive constants $L$ and $\beta$ such that if $t\mapsto x(t)$ is any solution of \eqref{lin} and $[a,b]\subset \mathcal{I}$, then for $a\leq t\leq b$:
    \begin{equation}\label{h-exp}
        |x(t)|\leq L\{h(t\ast a^{\ast-1})^{-\beta}|x(a)| + h(b\ast t^{\ast-1})^{-\beta}|x(b)|\}.
    \end{equation}
\end{definition}

\begin{remark}
\label{ACHE}
By using the alternative characterization of the uniform $h$--dichotomy provided by \eqref{alt-hd2} we also have an additional formulation for
\eqref{h-exp} as follows
\begin{displaymath}
        |x(t)|\leq L\{h(d(t,a))^{-\beta}|x(a)| + h(d(b,t))^{-\beta}|x(b)|\}
\end{displaymath}
\end{remark}

The thought--provoking article of K.J. Palmer \cite{Palmer} provides a characterization of the uniform $h$--dichotomies in terms of properties of uniform $h$--noncriticality and $h$--expansiveness restricted to the specific growth rate $h(t)=e^{t}$. A first attempt to generalize the Palmer's results has been carried out in \cite{E-P-R}. In fact it was proved that 
\begin{proposition} \cite[Th.1]{E-P-R}
\label{P1}
If the linear system \eqref{lin} has the properties of uniform bounded $h$--growth and uniform $h$--noncriticality on
$J^{+}:=[e_{\ast},+\infty)$ then it has a uniform $h$--dichotomy on $[e_{\ast},+\infty)$. 
\end{proposition}

The following result can be proved \textit{mutatis mutandi} from the proof of the previous result.

\begin{proposition} 
\label{P1.1}
If the linear system \eqref{lin} has the properties of uniform bounded $h$--decay and uniform $h$--noncriticality on
$J^{-}:=(a_{0},e_{\ast}]$ then it has a uniform $h$--dichotomy on $(a_{0},e_{\ast}]$. 
\end{proposition}

A direct consequence of Proposition \ref{P1} was the following equivalence:
\begin{proposition}\cite[Th.2]{E-P-R}
\label{P2}
    Assume that the linear system \eqref{lin} has a uniform bounded $h$--growth on $[e_{\ast},+\infty)$. Then the following three statements are equivalent:
    \begin{enumerate}
        \item[(i)] The system \eqref{lin} has a uniform $h$--dichotomy on $[e_{\ast},+\infty)$.
        \item[(ii)] The system \eqref{lin} is $h$--expansive on $[e_{\ast},+\infty)$.
        \item[(iii)] The system \eqref{lin} is uniformly $h$--noncritical on $[e_{\ast},+\infty)$.
    \end{enumerate}

    Moreover, without the assumption of uniform bounded $h$--growth, it is still true that the following implications are verified: $\textnormal{(i)}\Rightarrow\textnormal{(ii)}\Rightarrow\textnormal{(iii)}$.
\end{proposition}

\section{$h$--dichotomies on $J$}

The goal of this section is double, the first goal is to generalize the Proposition \ref{P2} for the case $\mathcal{I}=J$
whereas the second one is to provide necessary and sufficient conditions ensuring that if the property of uniform $h$--dichotomy is verified simultaneously on $J^{-}=(a_{0},e_{*}]$ and
$J^{+}=[e_{*},+\infty)$ then it is also verified on $J$. Both objectives will be achieved through several intermediate results.

With respect to the second goal, it is important to emphasize that if (\ref{lin}) has
a uniform $h$--dichotomy on the intervals $(a_{0},e_{*}]$ and $[e_{*},+\infty)$ this 
not implies that (\ref{lin}) has a uniform $h$--dichotomy on $J=(a_{0},e_{*}]\cup [e_{*},+\infty)$. In order to illustrate the above situation, let us consider the differential equation
\begin{equation}
\label{example}
x'=a(t)x  \quad \textnormal{for any $t\in J:=(0,+\infty)$},
\end{equation}
where $a\colon (0,+\infty)\to \mathbb{R}$ is defined as follows:
\begin{displaymath}
a(t)=\left\{\begin{array}{rcl}
\frac{1}{t}  &\textnormal{if} & 0<t\leq 1-\ell, \\\\
\phi(t)      &\textnormal{if} & 1-\ell< t < 1+\ell, \\\\
-\frac{1}{t} &\textnormal{if} & 1+\ell \leq t,
\end{array}\right.
\end{displaymath}
with $\ell\in (0,1)$ and $\phi \colon [1-\ell,1+\ell]\to \mathbb{R}$ is a continuous function such that
\begin{displaymath}
\phi(1-\ell)=\frac{1}{1-\ell} \quad \textnormal{and} \quad \phi(1+\ell)=-\frac{1}{1+\ell}. 
\end{displaymath}

It is straightforward to verify that
$$
\Phi(t)=\frac{1}{t} \quad \textnormal{and} \quad \Phi(t,s)=\left(\frac{t}{s}\right)^{-1}  \quad\textnormal{for any $t\geq s>1+\ell$},
$$
then, it is easy to deduce that the equation (\ref{example}) has an $h$--dichotomy on 
$\mathcal{I}_{1}=(1+\ell,+\infty)$ with $P=K=\alpha=1$ and the growth rate $h(t)=t$ defined on $J=(0,+\infty)$.

Similarly, we can verify that
$$
\Phi(t)=t \quad \textnormal{and} \quad \Phi(t,s)=\left(\frac{s}{t}\right)^{-1}=\frac{t}{s}  \quad\textnormal{for any $0<t\leq s<1-\ell$},
$$
and, as before, we can deduce that the equation (\ref{example}) has an $h$--dichotomy on 
$\mathcal{I}_{2}=(0,1-\ell)$ with $P=0$, $K=\alpha=1$ and the same growth rate $h(t)=t$.

As $h\colon (0,+\infty)\to (0,+\infty)$ is a growth rate, we can see that the unit element
of the corresponding group $(J,*)$ is $e_{*}=1$. Moreover, as $\ell \in (0,1)$ it follows that
 $\mathcal{I}_{1}\subset [e_{\ast},+\infty)$  and $\mathcal{I}_{2}\subset (0,e_{\ast}]$. By using \cite[Lemma 4.2]{E-P-R} we can see that (\ref{example})
 has an $h$ dichotomy on $(0,e_{\ast}]$ and $[e_{\ast},+\infty)$. Nevertheless, by using Lemma \ref{FDIA} we can deduce that
 the equation (\ref{example}) has not an $h$--dichotomy on $J=(0,+\infty)$ since any nontrivial solution is bounded
 on $J=(0,+\infty)$. In fact, notice that
 for any $x_{0}\neq 0$ we have that
$$
 t\mapsto x(t,t_{0},x_{0})=e^{\int_{t_{0}}^{t}a(s)\,ds}x_{0}
$$
is the solution of (\ref{example}) passing through $x_{0}$ at $t=t_{0}$. It is straightforward
to check that the boundedness follows since $\lim\limits_{t\to +\infty}x(t,t_{0},x_{0})=0$ and $\lim\limits_{t\to 0^{+}}x(t,t_{0},x_{0})=0$.

\subsection{Preliminary results}

The following results provide an alternative definition for a uniform $h$--dichotomy
on $J$ in terms of the properties of uniform $h$--dichotomies
on $(a,e_{*}]$ and $[e_{*},+\infty)$.

\begin{lemma}
\label{ISF}
If the linear system \eqref{lin} has a uniform $h$--dichotomy on $[e_{\ast},+\infty)$ with projectors $P^{+}(s)$ and constants
$(K,\alpha)$, then
\begin{equation}
\label{FW}
\mathcal{R}P^{+}(s)=\left\{\xi\in\mathbb{R}^n \colon \sup\limits_{t\geq s\geq e_{\ast}}|\Phi(t,s)\xi|<+\infty\right\}.
\end{equation}
\end{lemma}

\begin{proof}
Let $\xi\in\mathcal{R}P^{+}(s)$ with $s\geq e_{\ast}$. As the system (\ref{lin}) has a uniform $h$--dichotomy on $[e_{*},+\infty)$, for any $t\geq s\geq e_{\ast}$ we have that 
\begin{align*}
    |\Phi(t,s)\xi| &= |\Phi(t,s)P^{+}(s)\xi|\leq K\left(\frac{h(t)}{h(s)}\right)^{-\alpha}|\xi|,
\end{align*}
and, by using $h(t)\geq h(s)$ we easily deduce that 
\begin{equation}
\label{cota}
\sup\limits_{t\geq s\geq e_{\ast}}|\Phi(t,s)\xi|<+\infty.
\end{equation}

\smallskip
Now, let us consider $\xi\in \mathbb{R}^{n}$ such that (\ref{cota}) is fulfilled. By recalling that (\ref{lin}) has a uniform $h$--dichotomy on $[e_{*},+\infty)$ and using the identity:
$$
I-P^{+}(s)=\Phi(s,t)\Phi(t,s)[I-P^{+}(s)]=\Phi(s,t)[I-P^{+}(t)]\Phi(t,s),
$$
combined with $t\geq s>e_{\ast}$, we can deduce that
\begin{displaymath}
|[I-P^{+}(s)]\xi| \leq  K\left(\frac{h(t)}{h(s)}\right)^{-\alpha}|\Phi(t,s)\xi|.  
\end{displaymath}

By letting $t\to +\infty$ combined with (\ref{cota}) we can deduce that $P^{+}(s)\xi=\xi$,
which implies that $\xi \in \mathcal{R}P^{+}(s)$.
\end{proof}

\begin{lemma}
\label{ISB}
If the system \eqref{lin} has a uniform $h$--dichotomy on $(a_{0},e_{\ast}]$ with projectors $P^{-}(s)$ and constants $(K,\alpha)$, then
\begin{equation}
\label{BW}
\mathcal{N}P^{-}(s)=\left\{\xi\in\mathbb{R}^n \colon \sup\limits_{a_{0}<t\leq s\leq e_{\ast}}|\Phi(t,s)\xi|<+\infty\right\}.
\end{equation}
\end{lemma}

\begin{proof}
 Let and $\xi\in\mathcal{N}P^{-}(s)$ with $s\in (a_0, e_{\ast}]$. For any $a_0<t\leq s\leq e_{\ast}$ the uniform $h$--dichotomy leads to
$$|\Phi(t,s)\xi| = |\Phi(t,s)[I-P^{-}(s)]\xi|\leq K\left(\frac{h(s)}{h(t)}\right)^{-\alpha}|\xi|=K\left(\frac{h(t)}{h(s)}\right)^{\alpha}|\xi|<K|\xi|$$
since $h(t)\leq h(s)$ and we conclude that 
\begin{equation}
\label{cota2}
\sup\limits_{a_0<t\leq s\leq e_{\ast}}|\Phi(t,s)\xi|<+\infty.
\end{equation}

\smallskip
Now, let us consider $\xi\in \mathbb{R}^{n}$ such that (\ref{cota2}) is fulfilled. By using the identity
\begin{displaymath}
P^{-}(s)=P^{-}(s)\Phi(s,t)\Phi(t,s)=\Phi(s,t)P^{-}(t)\Phi(t,s),    
\end{displaymath}
combined with $a_{0}<t\leq s\leq e_{\ast}$ and the uniform $h$--dichotomy, we can deduce that
$$
|P^{-}(s)\xi|\leq K\left(\frac{h(s)}{h(t)}\right)^{-\alpha}|\Phi(t,s)\xi|.
$$

By letting $t\to a_{0}^{+}$ combined with (\ref{cota2}) and noticing that $h(t)\to 0$,
we can deduce that $P^{+}(s)\xi=0$,
which implies that $\xi \in \mathcal{N}P^{+}(s)$.
\end{proof}

\begin{remark}
\label{EXTPRO}
Notice that if the equation \eqref{lin} has a uniform $h$--dichotomy on $J^{+}=[e_{*},+\infty)$ with projector $P^{+}\colon J^{+}\to M_{n}(\mathbb{R})$ this projector can be extended to $J$ as follows:
\begin{equation}
\label{EP+}
\widetilde{P}^{+}(t)=\left\{\begin{array}{lcl}
P^{+}(t) & \textnormal{if} &  t\in (e_{*},+\infty), \\\\
\Phi(t,e_{*})P^{+}(e_{*})\Phi(e_{*},t) & \textnormal{if} & t\in (a_{0},e_{*}].
\end{array}\right.
\end{equation}

Furthermore, it is straightforward to prove that $\widetilde{P}^{+}(t)\widetilde{P}^{+}(t)=\widetilde{P}^{+}(t)$ and
$$
\widetilde{P}^{+}(t)\Phi(t,s)=\Phi(t,s)\widetilde{P}^{+}(s)  \quad \textnormal{for any $t,s\in J$}.
$$

Moreover, when \eqref{lin} has a uniform $h$--dichotomy on $J^{-}=(a_{0},e_{*}]$ with projector $P^{-}\colon J^{-}\to M_{n}(\mathbb{R})$ we can extend $P^{-}(\cdot)$ in a similar way.
\end{remark}

From now on, by an abuse of notation, the extensions of $P^{+}(\cdot)$ and $P^{-}(\cdot)$ to $J$ will be denoted
as $P^{+}(\cdot)$ and $P^{-}(\cdot)$ instead of  $\widetilde{P}^{+}(\cdot)$ and $\widetilde{P}^{-}(\cdot)$. The reader will 
not be disturbed by this fact.

\begin{lemma}
\label{LemPQvar}
Assume that the system \eqref{lin} has uniform $h$--dichotomies on both intervals $[e_{\ast},+\infty)$ and $(a_{0},e_{\ast}]$
with projectors $P^{+}(\cdot)$ and $P^{-}(\cdot)$ respectively which are defined on $J$ in the sense of Remark \ref{EXTPRO}. Then \eqref{lin} has no nontrivial bounded solution if and only if  
the projectors fulfill the condition:
\begin{equation}
\label{IdePr}
P^{+}(s)P^{-}(s)=P^{-}(s)P^{+}(s)=P^{+}(s) \quad \textnormal{ for any } s\in J.
\end{equation}
\end{lemma}

\begin{proof}

Firstly, let us assume that the unique bounded solution of the linear system \eqref{lin} is the trivial one, then by using lemmas \ref{ISF} and \ref{ISB} 
we can conclude that $\mathcal{R}P^{+}(s)\cap\mathcal{N}P^{-}(s)=\{0\}$, this fact combined with the identities 
$$
\mathbb{R}^n=\mathcal{R}P^{+}(s)\oplus\mathcal{N}P^{+}(s)=\mathcal{R}P^{-}(s)\oplus\mathcal{N}P^{-}(s)
$$
implies the existence of a subspace $V$ such that $\mathbb{R}^n=\mathcal{R}P^{+}(s)\oplus V\oplus \mathcal{N}P^{-}(s)$.  This allows us to take 
\begin{equation}
\label{sous-espaces}
\mathcal{N}P^{+}(s)=V\oplus \mathcal{N}P^{-}(s) \quad \textnormal{and} \quad \mathcal{R}P^{-}(s)=\mathcal{R}P^{+}(s)\oplus V
\end{equation}
and so $\mathcal{N}P^{-}(s)\subseteq\mathcal{N}P^{+}(s)$ and $\mathcal{R}P^{+}(s)\subseteq \mathcal{R}P^{-}(s)$.

Now, let $\xi\in\mathbb{R}^n$ such that  $\xi=\xi_1+\xi_2$ with $\xi_1\in\mathcal{R}P^{+}(s)$ and $\xi_2\in\mathcal{N}P^{+}(s)$. Then $P^{+}(s)\xi=\xi_1$ and, since $\xi_1\in \mathcal{R}P^{+}(s)\subset \mathcal{R}P^{-}(s)$, we get
$$
P^{-}(s)P^{+}(s)\xi=P^{-}(s)\xi_1=\xi_1=P^{+}(s)\xi,
$$
which shows that $P^{-}(s)P^{+}(s)=P^{+}(s)$ for any $s\in J$. 

Similarly, we can write $\xi=\xi_1+\xi_2 \in \mathbb{R}^{n}$ with $\xi_1\in\mathcal{R}P^{-}(s)$ and $\xi_2\in\mathcal{N}P^{-}(s)$, which leads to $P^{-}(s)\xi=\xi_1$. Moreover, as $\xi_{2}\in\mathcal{N}P^{-}(s)\subset \mathcal{N}P^{+}(s)$ we obtain
$$
P^{+}(s)P^{-}(s)\xi=P^{+}(s)\xi_1=P^{+}(s)\xi_1+P^{+}(s)\xi_2=P^{+}(s)\xi,
$$
and we have that $P^{+}(s)P^{-}(s)=P^{+}(s)$ for any $s\in J$. Finally, we conclude that $P^{-}(s)P^{+}(s)=P^{+}(s)P^{-}(s)=P^{+}(s)$ for any $s\in J$.

Conversely, suppose that the projections $P^{+}(s)$ and $P^{-}(s)$ for the dichotomies on  $[e_{\ast},+\infty)$ and $(a_{0},e_{\ast}]$ respectively can be chosen in such a way that the identities (\ref{IdePr}) are verified. Let $t\mapsto x(t)$ be a bounded solution of \eqref{lin}, then by lemmas \ref{ISF} and \ref{ISB} we can see that $x(s)\in\mathcal{R}P^{+}(s)\cap\mathcal{N}P^{-}(s)$, which allow us to deduce that $x(s)=P^{+}(s)x(s)=P^{+}(s)P^{-}(s)x(s)=P^{+}(s)0=0$ and so $x(t)$ is the trivial solution. This concludes the proof of the lemma. 
\end{proof}

\begin{remark}
\label{R33}
A careful reading of the above proof shows that the identity \eqref{sous-espaces} implies that
\begin{equation}
\label{contenciones}
\mathcal{N}P^{-}(s)\subseteq\mathcal{N}P^{+}(s) \quad \textnormal{and} \quad \mathcal{R}P^{+}(s)\subseteq \mathcal{R}P^{-}(s) \quad \textnormal{for any $s\in J$},
\end{equation}
which has interest on itself. In fact, an important question is to determine when \eqref{contenciones} are set identities or, equivalently, the subspace $V$ in \eqref{sous-espaces}
is $\{0\}$.
\end{remark}

\begin{lemma}\label{Lem-hexpanvar}
   Suppose that the system \eqref{lin} has  uniform $h$--dichotomies on the intervals $[e_{\ast},+\infty)$ and $(a_{0},e_{\ast}]$
   with projectors $P^{+}(s)$ and $P^{-}(s)$ respectively and  the unique bounded solution is the trivial one. Then \eqref{lin} is $h$--expansive on $J=(a_{0},+\infty)$. 
\end{lemma}

\begin{proof}
Without loss of generality it will be assumed that $K$ and $\alpha$ are the constants in both $h$--dichotomies. Moreover, given any non 
degenerate interval $[a,b]\subset J$, the proof will consider three cases: $e_{*}\leq a$, $e_{*}\geq b$ and $e_{*}\in [a,b]$.

\medskip 
\noindent \textit{Case 1: $e_{*}\geq b$}. Suppose that $t\mapsto x(t)$ is a solution of \eqref{lin} and $a_{0}<a\leq t\leq b\leq e_{\ast}$. 
By using $P^{-}(\cdot)+Q^{-}(\cdot)=I$, we have that
\begin{subequations}
  \begin{empheq}[left=\empheqlbrace]{align}
&
x(t)=\Phi(t,a)P^{-}(a)x(a)+\Phi(t,a)Q^{-}(a)x(a),  \label{HE1} \\
&
x(t)=\Phi(t,b)P^{-}(b)x(b)+\Phi(t,b)Q^{-}(b)x(b). \label{HE2}
 \end{empheq}
 \end{subequations}

By considering $t=b$ in (\ref{HE1}) we have 
\begin{displaymath}
x(b)=\Phi(b,a)P^{-}(a)x(a)+\Phi(b,a)Q^{-}(a)x(a),
\end{displaymath}
which replaced in the first right term of (\ref{HE2}) and recalling the invariance
 property (\ref{Invariance}) for $P^{-}(\cdot)$ leads to
\begin{displaymath}
\begin{array}{rcl}
x(t)&=&\Phi(t,b)P^{-}(b)\left\{\Phi(b,a)P^{-}(a)x(a)+\Phi(b,a)Q^{-}(a)x(a)\right\}+\Phi(t,b)Q^{-}(b)x(b)\\\\
&=&  \Phi(t,a)P^{-}(a)x(a)+\Phi(t,b)\Phi(b,a)P^{-}(a)Q^{-}(a)x(a)+\Phi(t,b)Q^{-}(b)x(b)\\\\
&=&  \Phi(t,a)P^{-}(a)x(a)+\Phi(t,b)Q^{-}(b)x(b).
\end{array}
\end{displaymath}

Now, by using the property of uniform $h$--dichotomy described by (\ref{alt-hd}) we can conclude that
\begin{equation}
\label{EE-MA2}
|x(t)|\leq  K\left\{h(t*a^{*-1})^{-\alpha}|x(a)|+h(b*t^{*-1})^{-\alpha}|x(b)|\right\},   
\end{equation}
and the $h$--expansiveness on $(a_{0},e_{*}]$ follows.

\noindent \textit{Case 2: $e_{*}\leq a$}. Suppose that $t\mapsto x(t)$ is a solution of \eqref{lin} and  $e_{\ast}\leq a\leq t\leq b$ then the estimation
    \begin{equation}
   \label{EE-MA}
    |x(t)|\leq K\left\{h(t\ast a^{\ast-1})^{-\alpha}|x(a)| + h(b\ast t^{\ast-1})^{-\alpha}|x(b)|\right\},
    \end{equation}
can be deduced analogously by considering $P^{+}(\cdot)$ instead of $P^{-}(\cdot)$ in (\ref{HE1})--(\ref{HE2})
and using $t=a$ in (\ref{HE2}).

\medskip

\medskip 
\noindent \textit{Case 3: $e_{*}\in [a,b]$}. As $t\mapsto x(t)$ is a solution of (\ref{lin}) by using the invariance properties (\ref{Invariance}) of the projectors $P^{-}(\cdot)$ and $P^{+}(\cdot)$, we can deduce that
\begin{subequations}
  \begin{empheq}[left=\empheqlbrace]{align}
&
P^{-}(a)x(a)=\Phi(a,e_{*})P^{-}(e_{*})x(e_{*}),  \label{EO1} \\
&
[I-P^{+}(b)]x(b)=\Phi(b,e_{*})[I-P^{+}(e_{*})]x(e_{*}). \label{EO2}
 \end{empheq}
 \end{subequations}

Moreover, as the unique bounded solution of (\ref{lin}) is the trivial one,  the identity $P^{+}(\cdot)=P^{+}(\cdot)P^{-}(\cdot)$ from Lemma \ref{LemPQvar} allow us 
to deduce that 
\begin{equation}
\label{itrip}
    x(e_{\ast})=P^{+}(e_{\ast})P^{-}(e_{\ast})x(e_{\ast})+[I-P^{+}(e_\ast)]x(e_{\ast}).
\end{equation}

As $a\leq e_{\ast}\leq b$, from the identities (\ref{EO1}),(\ref{EO2}) and (\ref{itrip}) combined with the identity $h(e_{\ast})=1$ and the dichotomy properties on
$(a_{0},e_{*}]$ and $[e_{*},+\infty)$, we can deduce that
   \begin{equation}
   \label{eunit}
    \begin{array}{rl}
        |x(e_{\ast})| &\leq |\Phi(e_{\ast},e_{\ast})P^{+}(e_{\ast})P^{-}(e_{\ast})x(e_{\ast})| + |[I-P^{+}(e_{\ast})]x(e_{\ast})|\\\\
        &\leq K|P^{-}(e_{\ast})x(e_{\ast})| + |[I-P^{+}(e_{\ast})]x(e_{\ast})|\\\\
        &= K|\Phi(e_{\ast},a)P^{-}(a)x(a)| + |\Phi(e_{\ast},b)[I-P^{+}(b)]x(b)|\\\\
        &\leq K^2(h(a))^{\alpha}|x(a)| + K(h(b))^{-\alpha}|x(b)|.
    \end{array}
\end{equation}
\noindent \textit{Subcase $a\leq e_{\ast}\leq t\leq b$}: By using a limit version of the Case 2 with $a=e_{*}$ and its corresponding estimation (\ref{EE-MA}) with $t\in [e_{*},b]$. It can be deduced that
the inequality (\ref{eunit}) combined with identity (\ref{group3}) imply that
    \begin{align*}
        |x(t)| &\leq K h(t\ast e^{\ast-1}_{\ast})^{-\alpha}|x(e_{\ast})| + Kh(b\ast t^{\ast-1})^{-\alpha}|x(b)|\\
        &\leq K(h(t))^{-\alpha}\left[K^2\left(\frac{1}{h(a)}\right)^{-\alpha}|x(a)|+K(h(b))^{-\alpha}|x(b)|\right] + Kh(b\ast t^{\ast-1})^{-\alpha}|x(b)|\\
        &= K^3\left(\frac{h(t)}{h(a)}\right)^{-\alpha}|x(a)| + K^2(h(t)h(b))^{-\alpha}|x(b)| + K\left(\frac{h(b)}{h(t)}\right)^{-\alpha}|x(b)|.
    \end{align*}

Moreover, by recalling that $e_{\ast}\leq t\leq b$ it follows that $1\leq h(t)\leq h(b)$ and so
    $$1\leq \frac{h(b)}{h(t)}\leq h(b)\leq h(t)h(b),$$
    which implies that 
    $$(h(t)h(b))^{-\alpha}\leq \left(\frac{h(b)}{h(t)}\right)^{-\alpha}.$$
    Thus, for $a\leq e_{\ast}\leq t\leq b$ it is concluded that
    \begin{displaymath}
    \begin{array}{rcl}
        |x(t)| &\leq&  \displaystyle K^3\left(\frac{h(t)}{h(a)}\right)^{-\alpha}|x(a)| + (K^2+K)\left(\frac{h(b)}{h(t)}\right)^{-\alpha}|x(b)| \\\
        &\leq& \displaystyle \tilde{K}\left\{h(t*a^{*-1})^{-\alpha}|x(a)|+
        h(b*t^{*-1})^{-\alpha}|x(b)|\right\},
        \end{array}
    \end{displaymath}
where $\tilde{K}= \max\{K^{3},K+K^{2}\}$
and the $h$--expansiveness follows.

\medskip
\noindent \textit{Subcase $a_{0}<a\leq t\leq e_{\ast}\leq b$}: Similarly, by using a limit version of the Case 1 with $b=e_{*}$ and its corresponding estimation (\ref{EE-MA2}) with $t\in [a,e_{*}]$. It can be deduced that
the (\ref{eunit}) and (\ref{group3}) imply that
    \begin{align*}
        |x(t)| &\leq K h(t\ast a^{\ast-1})^{-\alpha}|x(a)| + Kh(e_{\ast}\ast t^{\ast-1})^{-\alpha}|x(e_{\ast})|\\
        &\leq K\left(\frac{h(t)}{h(a)}\right)^{-\alpha}|x(a)| + K\left(\frac{1}{h(t)}\right)^{-\alpha}\left[K^2\left(\frac{1}{h(a)}\right)^{-\alpha}|x(a)| + K(h(b))^{-\alpha}|x(b)|\right]\\
        &= K\left(\frac{h(t)}{h(a)}\right)^{-\alpha}|x(a)| + K^3\left(\frac{1}{h(t)h(a)}\right)^{-\alpha}|x(a)| + K^2\left(\frac{h(b)}{h(t)}\right)^{-\alpha}|x(b)|.
    \end{align*}
    As $a\leq t\leq e_{\ast}$ then $0<h(a)\leq h(t)\leq 1$ and so
    $$h(a)h(t)\leq h(a)\leq\frac{h(a)}{h(t)}\leq 1.$$
    It follows that 
    $$\frac{1}{h(a)h(t)}\geq\frac{h(t)}{h(a)}\geq1,$$
    which in turn implies that
    $$\left(\frac{1}{h(a)h(t)}\right)^{-\alpha}\leq\left( \frac{h(t)}{h(a)}\right)^{-\alpha}.$$
        Thus, for $a\leq t\leq e_{\ast}\leq b$ it is concluded that
    \begin{displaymath}
    \begin{array}{rcl}
        |x(t)| &\leq & \displaystyle (K+K^3)\left(\frac{h(t)}{h(a)}\right)^{-\alpha}|x(a)| + K^2\left(\frac{h(b)}{h(t)}\right)^{-\alpha}|x(b)|\\\\
   &\leq & \displaystyle K'\left\{h(t*a^{*-1})^{-\alpha}|x(a)|+ h(b*t^{*-1})^{-\alpha}|x(b)|\right\},     
        \end{array}
    \end{displaymath}
where $K'=\max\{K+K^{3},K^{2}\}$ and the $h$--expansiveness follows by gathering all the above cases.
\end{proof}

\subsection{Generalization of Proposition \ref{P2} to $\mathcal{I}=J$.}

\begin{theorem}
\label{PJ}
    If the system \eqref{lin} has a uniform bounded $h$--growth on $[e_{\ast},+\infty)$ and a uniform bounded $h$--decay on $(a_{0},e_{\ast}]$. Then the following three statements are equivalent:
    \begin{enumerate}
        \item[(i)] The system \eqref{lin} has uniform $h$--dichotomies on $[e_{\ast},+\infty)$ and $(a_{0},e_{\ast}]$ and has no nontrivial bounded solution.
        \item[(ii)] The system \eqref{lin} is $h$--expansive on $J$.
        \item[(iii)] The system \eqref{lin} is uniformly $h$--noncritical on $J$.
    \end{enumerate}
    Moreover, without the assumption of uniform bounded $h$--growth and $h$--decay, it is still true that the following implications are verified: \textnormal{(i)}$\Rightarrow$ \textnormal{(ii)} $\Rightarrow$ \textnormal{(iii)}.
\end{theorem}

\begin{proof} The implication  (i) $\Rightarrow$ (ii) follows directly from the Lemma \ref{Lem-hexpanvar} recently proved. On the other hand,
despite the implication (ii) $\Rightarrow$ (iii) can be proved similarly as in the proof of Theorem 2 in \cite{E-P-R}, it will be useful to do it again in order to make this article the most self-contained as possible. In fact, if $t\mapsto x(t)$ is a solution of (\ref{lin}), there exists positive constants $L$ and $\beta$ and $[a,b]\subset J$ such that (\ref{h-exp}) is verified for any $t\in [a,b]$.

As $h(t)\to +\infty$ and $\beta>0$, we can choose $T>e_{*}$ such that $\theta:=2Lh(T)^{-\beta}<1$ and, by using (\ref{order})--(\ref{order2}), we can deduce that
$e_{*}>T^{*-1}$ and $t \in (t*T^{*-1},t*T)$ for any $t\in J$.

By considering $a=t*T^{*-1}$ and $b=t*T$, the inequality (\ref{h-exp}) becomes
\begin{displaymath}
\begin{array}{rcl}
|x(t)| &\leq & L\left\{h(T)^{-\beta}|x(t*T^{*-1})|+h(T)^{-\beta}|x(t*T)|\right\} \\\\
&\leq &  \theta \sup\limits_{u\in [t*T^{-1},t*T]}|x(u)|,
\end{array}
\end{displaymath}
and the property of uniform $h$--noncriticality on $J$ is fulfilled.

    Finally, in order to proof the implication (iii)$\Rightarrow$(i), we will suppose that \eqref{lin} is uniformly $h$--noncritical on $J$. Then, by Remark \ref{ANCU}, the system \eqref{lin} is uniformly $h$--noncritical in both $[e_{*},+\infty)$ and
    $(a_{0},e_{*}]$. Now, as \eqref{lin} has a uniform bounded $h$--growth on $[e_{\ast},+\infty)$ and a uniform bounded $h$--decay on $(a_{0},e_{\ast}]$ it follows by Propositions \ref{P1} and \ref{P1.1} that \eqref{lin} has uniform $h$--dichotomies on $[e_{\ast},+\infty)$ and $(a_{0},e_{\ast}]$ respectively.

    Now let $t\mapsto x(t)$ be a bounded solution of \eqref{lin} such that 
    $$\|x\|:=\sup\limits_{t\in J}|x(t)|<+\infty.$$

As (iii) is verified, by using Definition \ref{NCUTR} with $\mathcal{I}=J$ it follows that for any $t\in J$ we have the existence of $\theta\in (0,1)$ and $T>e_{\ast}$ such that 
    $$|x(t)|\leq\theta\sup\limits_{|u\ast t^{\ast-1}|_{\ast}\leq T}|x(u)|\leq \theta\|x\|,$$
    which implies that $\|x\|\leq\theta\|x\|$. Hence, $\|x\|=0$ and so $x(t)=0$ for all $t\in J$. This concludes that \eqref{lin} has no nontrivial solution and the proof of the theorem is complete.
\end{proof}

\subsection{Uniform $h$--dichotomy on $J$}

As we have seen at the beginning of this section, if the linear system (\ref{lin}) has simultaneously the properties of uniform $h$-dichotomy on the half lines $J^{-}$ and $J^{+}$ this does not always implies the existence of  an $h$--dichotomy on the full line $J$. Now, we will provide a necessary and sufficient condition ensuring the uniform $h$--dichotomy on $J$. In this context, we recall the following definition adapted from \cite[p.536]{Battelli} to a general uniform $h$--dichotomy. 
\begin{definition}
\label{index}
 The \textit{index} of a $n$--dimensional linear system $\dot{x}=A(t)x$ with uniform $h$--dichotomies both on $J^{-}=(a_{0},e_{*}]$ and $J^{+}=[e_{*},+\infty)$ respectively, is given by
 $$
 i(A)=\textnormal{\text{dim}}\,\mathcal{R}P^{+}(\cdot)+\textnormal{\text{dim}}\,\mathcal{N} P^{-}(\cdot)-n,
 $$
where $P^{-}(\cdot)$ and $P^{+}(\cdot)$ are the projections for the dichotomies on $J^{-}$ and $J^{+}$ respectively.
\end{definition}

We point out that as $P^{-}(\cdot)$ and $P^{+}(\cdot)$ fulfills the invariance property (\ref{Invariance})
on $J^{-}$ and $J^{+}$ respectively then the dimensions are invariant.

The following result generalizes Proposition 1 from \cite{Battelli} for the uniform $h$--dichotomy:
\begin{theorem}
\label{FullDico}
The system \eqref{lin} has a uniform $h$--dichotomy on $J$
if and only if: 
\begin{itemize}
\item[a)] The system has uniform $h$--dichotomies on $J^{-}$ and $J^{+}$
with projections $P^{-}(\cdot)$ and $P^{+}(\cdot)$ respectively,
\item[b)] The system has no nontrivial bounded solutions, 
\item[c)] The index of the linear system is $i(A)=0$.
\end{itemize}
 \end{theorem}

\begin{proof}
Firstly, if the linear system has a uniform $h$--dichotomy on $J$ with projector $P(\cdot)$,
the properties a) and c) are verified straightforwardly with $P(\cdot)=P^{-}(\cdot)=P^{+}(\cdot)$. In addition,
the property b) follows from Lemma \ref{FDIA}.

Secondly, let us assume that the properties a),b) and c) are fulfilled. Now, we will prove that:
\begin{equation}
\label{esgalite-sub}
\mathcal{N} P^{-}(t) = \mathcal{N} P^{+}(t) \quad \textnormal{and} \quad  \mathcal{R} P^{+}(t) = \mathcal{R} P^{-}(t).
\end{equation}

In order to prove the above identity, notice that by a) and b) we can use the Lemma \ref{LemPQvar} and Remark \ref{R33}, which implies the contentions
$\mathcal{N}P^{-}(t)\subseteq \mathcal{N}P^{+}(t)$ and  $\mathcal{R}P^{+}(s)\subseteq \mathcal{R}P^{-}(s)$.

Now, if (\ref{esgalite-sub}) is not verified, we will have that $\mathcal{N} P^{-}(t) \subset \mathcal{N} P^{+}(t)$ 
and/or $\mathcal{R}P^{+}(t)\subset \mathcal{R}P^{-}(t)$.

In case that  $\mathcal{N} P^{-}(t) \subset \mathcal{N} P^{+}(t)$,  we have 
$\text{dim}\, \mathcal{N} P^{-}(t)<\text{dim} \,\mathcal{N} P^{+}(t)$, which combined with $i(A)=0$ leads to
\begin{displaymath}
\begin{array}{rcl}
n &=& \text{dim}\, \mathcal{N} P^{-}(t) + \text{dim} \, \mathcal{R}\, P^{+}(t) \\
&<&  \text{dim}\, \mathcal{N} P^{+}(t) + \text{dim} \,\mathcal{R}\, P^{+}(t)=n,
\end{array}
\end{displaymath}
obtaining a contradiction and   $\mathcal{N} P^{-}(t)=\mathcal{N} P^{+}(t)$ follows. Moreover, the 
identity $\mathcal{R}\, P^{+}(t) = \mathcal{R}\, P^{-}(t)$ can be proved in
a similar way.

As $P^{-}(t)\xi\in \mathcal{R} P^{-}(t)=\mathcal{R} P^{+}(t)$ for any $\xi\in \mathbb{R}^{n}$ we can easily conclude that
$P^{+}(t)P^{-}(t)\xi=P^{-}(t)\xi$. Moreover, by using the statements a) and b) combined with Lemma \ref{LemPQvar}
and the identity (\ref{IdePr}) we also have that $P^{+}(t)P^{-}(t)\xi=P^{+}(t)\xi$. By gathering these identities we can see that
$$
P^{-}(t)\xi = P^{+}(t)\xi \quad \textnormal{for any $\xi \in \mathbb{R}^{n}$},
$$
and the identity $P^{+}(t)=P^{-}(t)$ follows, this implies that the linear system has
a uniform $h$-dichotomy on $J$ with the same projector $P(\cdot)=P^{+}(\cdot)=P^{-}(\cdot)$.
\end{proof}

The above result allow us to revisit the scalar equation $x'=a(t)x$ described by (\ref{example})  at the beginning
of this section 3 with $J=(0,+\infty)$, $e_{*}=1$ and $h(t)=t$,  where the equation has $h$ dichotomies on $(0,1]$
and $[1,+\infty)$ with projectors $P^{-}(t)=0$ and $P^{+}(t)=1$ respectively. As $\mathcal{R}P^{+}=\mathcal{N}P^{-}=\mathbb{R}$,
it follows that $i(a)=1$ and the $h$--dichotomy on $J$ is not verified.

\section{Generalized Floquet theory  and uniform $h$--dichotomy}
In \cite[p.180]{Burton}, T.A. Burton and J.S. Muldowney proposed a generalization of the Floquet theory for
(\ref{lin}) by coining the following definition tailored for our purposes:
\begin{definition}
\label{GFlo}
The linear system \eqref{lin} is a generalized Floquet system with respect to $f$ if $f$ is an absolutely
continuous function on $J:=(a_{0},+\infty)$ such that 
\begin{equation}
\label{BM1}
f'(t)A(f(t))=A(t)
\end{equation}
for almost all $t>a_{0}$ and
\begin{equation}
\label{BM2}
f(t)>t
\end{equation}
for all $t>a_{0}$.
\end{definition}

In this section we will work with the specific function $f(t)=t*T$, where the group
$(J,*)$ is defined in terms of a differentiable growth rate  $h(\cdot)$ and $T>e_{*}$. In addition, we will assume
that the matrix valued function $A(\cdot)$ from the system (\ref{lin})
fulfills the following property:
\begin{equation}
\label{GFS}
(t*T)'A(t*T) = \frac{h'(t)h(T)}{h'(t*T)}A(t*T)=A(t)  \quad \textnormal{for any $t\in J$},
\end{equation}
where the first identity follows by $h(t*T)=h(t)h(T)$ from  (\ref{group0}) which implies that $h'(t*T)(t*T)'=h'(t)h(T)$. Moreover, $t*T>t$ for any $t\in J$ since $T>e_{*}$. Notice that if $h(t)=e^{t}$ then 
$(J,*)=(\mathbb{R},+)$ and (\ref{GFS}) becomes $A(t+T)=A(t)$ for any $t\in \mathbb{R}$.

In \cite{Burton} it is shown that the properties (\ref{BM1})--(\ref{BM2}) allow to emulate
the construction of the monodromy matrix similarly as in the classical Floquet's theory \cite[Ch.1]{Adrianova},\cite[pp.118--121]{Hale}. Obviously,
these results can be applied for the particular case of Definition \ref{GFlo} given by the condition (\ref{GFS}). Nevertheless, we will see that
our restriction to this specific case combined with the underlying properties of the group $(J,*)$ will provide us a criterion for detect the uniform $h$--dichotomy.

\subsection{Technical results}
In order to make this section self contained, we will revisit some results from \cite{Burton} tailored for the assumption (\ref{GFS}).

\begin{lemma}
If the linear system \eqref{lin} fulfills the property \eqref{GFS} and $\Phi(t)$ is a fundamental matrix of \eqref{lin} then
$\Psi(t):=\Phi(t*T)$ is also a fundamental matrix of \eqref{lin}.
\end{lemma}

\begin{proof}
By using the fact that $\Phi(\cdot)$ is a fundamental matrix combined with the chain rule and (\ref{GFS}) it follows that
\begin{displaymath}
\begin{array}{rcl}
\Psi'(t) &=& (t*T)'\Phi'(t*T) \\\\
&=& (t*T)'A(t*T)\Phi(t*T) \\\\
&=& A(t)\Psi(t).
\end{array}
\end{displaymath}
\end{proof}

As the columns of $\Phi(t)$ and $\Phi(t*T)$ conform basis of solutions of (\ref{lin}), there exists
a non singular matrix $V$ such that
\begin{equation}
\label{mono1}
\Phi(t*T)=\Phi(t)V  \quad \textnormal{for any $t\in J$}.
\end{equation}

Similarly as in the classical Floquet's theory, the matrix $V$ will be called as the \textit{monodromy matrix}.
The eigenvalues of $V$ will be called as the $h$--Floquet multipliers.

\begin{remark}
\label{remarque-Floquet}
Some direct consequences of \eqref{mono1} are:
\begin{itemize}
\item[i)] The "biperiodicity" property
\begin{displaymath}
\Phi(t*T,s*T)=\Phi(t,s) \quad   \textnormal{for any $t,s\in J$}.  
\end{displaymath}
\item[ii)] The property
\begin{equation}
\label{mono-n}
\Phi(t*T^{*n})=\Phi(t)V^{n} \quad \textnormal{for any $n\in \mathbb{N}$}.
\end{equation}
\item[iii)] If $\Phi(e_{*})=I$, then it follows that $\Phi(T)=V$. 
\end{itemize}
\end{remark}

\begin{lemma}
\label{TF3}
Under the assumption \eqref{GFS}, a number $\rho \in \mathbb{C}$ is an $h$--Floquet's multiplier of the system \eqref{lin} if and only if there exists a nontrivial solution $t\mapsto x(t)$ such that
\begin{equation}
\label{Floquet4}
x(t*T)=\rho \,x(t) \quad \textnormal{for any $t\in J$}
\end{equation}
\end{lemma}

\begin{proof}
Without loss of generality, it can be assumed that $\Phi(e_{*})=I$, which leads to $\Phi(T)=V$ according to Remark \ref{remarque-Floquet}.

Firstly, it will be assumed that $\rho$ is an eigenvalue of the monodromy matrix $\Phi(T)=V$. Then, by definition we have the existence of $\xi\in \mathbb{C}^{n}\setminus \{0\}$ such 
that $\Phi(T)\xi=\rho \, \xi$. Now, let us consider the solution $t\mapsto x(t)$ of the initial value problem
\begin{displaymath}
\left\{\begin{array}{rcl}
x'&=&A(t)x \\
x(e_{*})&=&\xi.
\end{array}\right.
\end{displaymath}

It is straightforward to see that $x(t)=\Phi(t)\xi$ for any $t\in J$, and we can see that (\ref{mono1}) implies that
$$
x(t*T)=\Phi(t*T)\xi=\Phi(t)V\xi=\rho\, x(t)
$$
and the identity (\ref{Floquet4}) follows.

Secondly, let us assume that (\ref{Floquet4}) is verified. By evaluating this identity at $t=e_{*}$ we obtain that
$$
x(T)=\rho \, x(e_{*}).
$$

On the other hand, as any non trivial solution $t\mapsto x(t)$ of (\ref{lin}) is given by $x(t)=\Phi(t)x(e_{*})$, this allows to 
see that
$$
x(T)=\Phi(T)x(e_{*})=V\,x(e_{*}).
$$

As the right parts of the above identities are equal to $x(T)$ we have the identity
$$
Vx(e_{*})=\rho\, x(e_{*}) 
$$
and it follows that $\rho$ is an eigenvalue of $V$. 
\end{proof}

\subsection{Uniform $h$--dichotomy for generalized Floquet systems}
\begin{theorem}
\label{FDE}
If the linear system \eqref{lin} fulfills the property \eqref{GFS} the following properties
are equivalent:
\begin{itemize}
\item[i)] The eigenvalues of the monodromy matrix $V$ are not in the unit circle, 
\item[ii)] The linear system \eqref{lin} has a uniform $h$--dichotomy on $J$.
\end{itemize}
\end{theorem}

\begin{proof}
If we assume that the property i) is fulfilled, we can study the difference equation
\begin{equation}
\label{difference}
x_{n+1}=Vx_{n}.
\end{equation}

As the eigenvalues of $V$ are not in the unit circle, it can be proved (see \textit{e.g.} \cite{PS}) that the autonomous system (\ref{difference})
has an exponential dichotomy on $\mathbb{Z}$, that is, there exists a constant projector $P^{2}=P$
and positive constants $K_{0}$ and $\alpha$ such that
\begin{equation}
\label{difference2}
\left\{ \begin{array}{rcl}
||V^{n}P[V^{k}]^{-1}||  &\leq & K_{0}e^{-\alpha(n-k)} \quad \textnormal{for any $n\geq k$}\\\\
||V^{n}(I-P)[V^{k}]^{-1}||  &\leq & K_{0}e^{-\alpha(k-n)}  \quad \textnormal{for any $k\geq n$}.
\end{array}\right.
\end{equation}

As it was stated in the subsection 2.2, we can construct the uniform partition for $J$:
$$
J=\bigcup\limits_{n\in \mathbb{Z}}I_{n}  \quad \textnormal{where $J_{n}=[T^{\ast (n-1)},T^{\ast n})$}.
$$

Now, when considering a pair $(t,s)$ with $t>s$, there exists a pair of integers $n,k$ with $n\geq k$ such that 
\begin{equation}
\label{flo4}
s\in [T^{*(k-1)},T^{*k}) \quad \textnormal{and} \quad t\in [T^{*(n-1)},T^{*n}). 
\end{equation}

As $h(\cdot)$ is strictly increasing the properties (\ref{flo4}) combined with (\ref{LP}) implies that
$$
h(T^{*(k-1)})\leq h(s) < h(T^{*k}) \iff h(T)^{k-1}\leq h(s) < h(T)^{k},
$$
then we can deduce that
\begin{equation}
\label{ineg-k}
k-1\leq \frac{\ln (h(s))}{\ln(h(T))}<k,
\end{equation}
whereas the inequality
\begin{equation}
\label{ineg-n}
n-1\leq \frac{\ln (h(t))}{\ln(h(T))}<n,
\end{equation}
can be deduced similarly. As $\alpha>0$, the estimations (\ref{ineg-k})--(\ref{ineg-n}) imply that
\begin{equation}
\label{flo5}
\alpha(k-n)<\alpha + \frac{\alpha}{\ln(h(T))}\ln \left(\frac{h(s)}{h(t)}\right)=\alpha+\ln\left(\frac{h(t)}{h(s)}\right)^{-\frac{\alpha}{\ln(h(T))}}. 
\end{equation}

Now, by (\ref{flo4}) we can deduce that
\begin{equation}
\label{flo6}
s=T^{*(k-1)}*\tilde{s} \quad \textnormal{and} \quad t=T^{*(n-1)}*\tilde{t} \quad \textnormal{for some $\tilde{t},\tilde{s}\in [e_{*},T)$}.
\end{equation}

By using (\ref{mono-n}) we can deduce that
\begin{displaymath}
\Phi(t)=\Phi(\tilde{t}*T^{*(n-1)})=\Phi(\tilde{t})V^{n-1} \quad \textnormal{and} \quad   \Phi(s)=\Phi(\tilde{s}*T^{*(k-1)})=\Phi(\tilde{s})V^{k-1},  
\end{displaymath}
which combined with (\ref{difference2}) and (\ref{flo5}) leads to
\begin{displaymath}
\begin{array}{rcl}
||\Phi(t)P\Phi^{-1}(s)||&=&||\Phi(\tilde{t})V^{n-1}P[V^{k-1}]^{-1}\Phi^{-1}(\tilde{s})|| \\
&\leq  &   K_{0}K_{1}K_{2}e^{-\alpha(n-k)}  \\
&\leq  &   K_{0}K_{1}K_{2}e^{\alpha}\left(\frac{h(t)}{h(s)}\right)^{-\frac{\alpha}{\ln(h(T))}},
\end{array}
\end{displaymath}
where
$$
K_{1}=\max\limits_{u\in [e_{*},T]}||\Phi(u)|| \quad \textnormal{and} \quad K_{2}=\max\limits_{u\in [e_{*},T]}||\Phi^{-1}(u)||.
$$

The case $t<s$ can be addressed similarly and it can be obtained that
\begin{equation}
\label{difference3}
\left\{ \begin{array}{rcl}
||\Phi(t)P\Phi^{-1}(s)||  &\leq & K\left(\frac{h(t)}{h(s)}\right)^{-\tilde{\alpha}} \quad \textnormal{for any $t\geq s$ with $t,s\in J$}\\\\
||\Phi(t)(I-P)\Phi^{-1}(s)||  &\leq & K \left(\frac{h(s)}{h(t)}\right)^{-\tilde{\alpha}} \quad \textnormal{for any $s\geq t$ with $t,s\in J$},
\end{array}\right.
\end{equation}
with $K=K_{0}K_{1}K_{2}e^{\alpha}$ and $\tilde{\alpha}=\alpha/\ln(h(T))$. Now, by using Proposition \ref{ProyCte} we can conclude that the uniform $h$--dichotomy on $J$ follows.

If we assume that the linear system (\ref{lin}) has a uniform $h$--dichotomy, the property i) can be verified by contradiction. 

In fact, if $\rho$ is an $h$--Floquet multiplier with $|\rho|=1$ by Lemma \ref{TF3} we have the existence of
a non trivial solution $t\mapsto x(t)=\Phi(t)\xi$ of (\ref{lin}) such that (\ref{Floquet4}) is verified with
$V\xi=\rho \xi$.

If $t>e_{*}$, there exists $n\in \mathbb{N}$ such that $t\in [T^{*(n-1)},T^{*n})$, then there exists $\tilde{t}\in [e_{*},T)$
such that
$$
x(t)=\Phi(t)\xi=\Phi(T^{*(n-1)}*\tilde{t})\xi=\Phi(\tilde{t})V^{n}\xi=\Phi(\tilde{t})\rho^{n}\xi,
$$
which implies that 
\begin{equation}
\label{acotamiento}
|x(t)|\leq M_{\xi}:=\max\limits_{u\in [e_{*},T]}||\Phi(u)|||\xi| 
\end{equation}
is fulfilled for any $t\geq e_{*}$

We will see now that the bound (\ref{acotamiento}) is fulfilled for any $t\in J$. In fact, if $t\in (T^{*-1},e_{*}]$, then
it follows that $e_{*}<t*T<T$. Then, by using (\ref{Floquet4}) combined with $|\rho|=1$ it follows that $|x(t)|=|x(t*T)|\leq M_{\xi}$
and we can proceed recursively. 

As (\ref{acotamiento}) is verified for any $t\in J$, this implies that $t\mapsto x(t)$ is bounded on $J$, then the 
statement b) from Theorem \ref{FullDico} would imply that $t\mapsto x(t)$ is the trivial solution, obtaining a contradiction and the Theorem follows.
\end{proof}

\begin{corollary}
If the linear system \eqref{lin} fulfills the property \eqref{GFS} and $t\mapsto x(t,t_{0},x_{0})$ is a nontrivial solution of \eqref{lin} passing trough $x_{0}\neq 0$ at $t=t_{0}\in J$,
then 
\begin{itemize}
\item[a)] If all the $h$--Floquet multipliers are inside the unit circle then the origin is uniformly $h$--stable, that is, there exists $K>0$ and $\alpha>0$ such that
\begin{equation}
\label{contract-flo}
|x(t,t_{0},x_{0})|\leq K\left(\frac{h(t)}{h(t_{0})}\right)^{-\alpha}|x_{0}|  \quad \textnormal{for any $t\geq t_{0}$}.  
\end{equation}
\item[b)] If all the $h$--Floquet multipliers are outside the unit circle then the origin is uniformly $h$--unstable, that is,
there exists $K>0$ and $\alpha>0$ such that
\begin{equation}
\label{exand-flo}
|x(t,t_{0},x_{0})|\geq \left(\frac{h(t)}{h(t_{0})}\right)^{\alpha}\frac{|x_{0}|}{K}  \quad \textnormal{for any $t\geq t_{0}$}.  
\end{equation}
\end{itemize}
\end{corollary}

\begin{proof}
As the $h$--Floquet multipliers are either inside or outside the unit circle, it follows by Theorem
\ref{FDE} that the system (\ref{lin}) has a uniform $h$--dichotomy. 

If the monodromy matrix $V$ has all its eigenvalues inside the unit circle, it is easy to see that
$P=I$ in \eqref{difference2} and, by following the lines of the proof of Theorem \ref{FDE}, we can conclude 
that the system (\ref{lin}) has a uniform $h$--dichotomy with projector $P=I$. Now, the statement a) follows by using Lemma \ref{SPLIT} combined with the left estimation of (\ref{ae1}).

Similarly, if the monodromy matrix $V$ has all its eigenvalues outside the unit circle, it is easy to see that
$P=0$ in \eqref{difference2} and by following the lines of the proof of Theorem \ref{FDE} we have 
that the system (\ref{lin}) has a uniform $h$--dichotomy with projector $P=0$. Now, the statement a) follows by using Lemma \ref{SPLIT} combined with the right estimation of (\ref{ae1}).
\end{proof}

The above result generalizes Theorem 1 from \cite[p.190]{Burton} which states that the solutions of a generalized Floquet
system (\ref{lin}) are either asymptotically stable or unstable according to the location of the $h$-Floquet multipliers
either inside or outside the unit circle. The originality of our result is that the properties of the group $(J,*)$ allow us to provide
a sharper characterization of the contractions (\ref{contract-flo}) and expansions (\ref{exand-flo}) by noticing that they are dominated
by the growth rate $h(\cdot)$ associated to the group $(J,*)$.

\subsection{An example}
Given $\alpha>0$, let us consider the diagonal system
\begin{displaymath}
\left(\begin{array}{c}
\dot{x}_{1}\\
\dot{x}_{2}
\end{array}\right)=
A(t)
\left(\begin{array}{c}
x_{1}\\
x_{2}
\end{array}\right) \quad \textnormal{with} \quad 
A(t)=\left[\begin{array}{cc}
\displaystyle -\alpha \frac{1}{t}  &  0 \\
0 & \displaystyle \alpha \frac{1}{t}
\end{array}\right]
\quad \textnormal{for any $t\in J=(0,+\infty)$},
\end{displaymath}
which is a particular case of the system (\ref{exemple}) studied in section 2 by considering $J=(0,+\infty)$ and 
the growth rate $h\colon J\to J$ defined by $h(t)=t$.

The corresponding group
operation for this growth rate is $t*s=ts$ and the unit element is $e_{*}=1$ which, given any $T>e_{*}$, allows to deduce that 
\begin{displaymath}
\frac{h'(t)h(T)}{h'(t*T)}A(t*T)=T\left[\begin{array}{cc}
\displaystyle -\alpha \frac{1}{tT}  &  0 \\
0 & \displaystyle \alpha \frac{1}{tT}
\end{array}\right]=\left[\begin{array}{cc}
\displaystyle -\alpha \frac{1}{t}  &  0 \\
0 & \displaystyle \alpha \frac{1}{t}
\end{array}\right]=A(t),
\end{displaymath}
and  the identity
(\ref{GFS}) follows for any $t\in J$.

When considering a fundamental matrix
$\Phi(t)=\textnormal{Diag}\{t^{-\alpha},t^{\alpha}\}$, we can see that $\Phi(e_{*})=I$
and
\begin{displaymath}
\Phi(t*T^{*n})=
\left[\begin{array}{cc}
\frac{1}{(tT^{n})^{\alpha}}  &  0 \\
0 & (tT^{n})^{\alpha}
\end{array}\right]=\left[\begin{array}{cc}
\frac{1}{t^{\alpha}}  &  0 \\
0 & t^{\alpha}
\end{array}\right] \left[\begin{array}{cc}
\frac{1}{T^{\alpha}}  &  0 \\
0 & T^{\alpha}
\end{array}\right]^{n}=\Phi(t)V^{n}
\end{displaymath}
and the identity (\ref{mono-n}) is verified. Moreover, as $T>1$ and $\alpha>0$ it follows that
the eigenvalues of $V$ are $T^{-\alpha}<1<T^{\alpha}$. By Theorem \ref{FDE} we have that the diagonal
system has a uniform $h$--dichotomy on $J=(0,+\infty)$, which is consistent with the more general example studied at the beginning of section 2.
Finally, notice that the difference system $x_{n+1}=Vx_{n}$ has an exponential dichotomy on $\mathbb{Z}$ and the projector
$P(t)$ coincides with the one stated in (\ref{dugma}).

\section*{Declaration of interests}

The authors declare that they have no known competing financial
interests or personal internships that could have appeared to influence the work reported
on the paper.

\section*{Declaration of availability of data}

No data was used for the research described in the article.

\end{document}